\documentclass[
final,
nomarks
]{dmtcs-episciences}

\usepackage[utf8]{inputenc}
\usepackage{subfigure}
\usepackage{amsmath}
\usepackage{amssymb}
\usepackage{amsthm}
\usepackage{color}
\usepackage{verbatim}
\usepackage[title,toc,titletoc]{appendix}
\usepackage{hyperref}
\usepackage{colortbl}
\usepackage{makecell}
\hypersetup{
	colorlinks,
	citecolor=blue,
	linkcolor=blue,
	urlcolor=black}

\newtheorem{obs}{Observation}
\newtheorem{thm}{Theorem}
\newtheorem{prp}{Property}

\newtheorem{cor}{Corollary}
\newtheorem{rem}{Remark}
\newtheorem{exm}{Example}

\usepackage[round]{natbib}

\author[Tatjana Zec et. al]{Tatjana Zec
  \and Milana Grbi\'c 
}
\title[Several Roman domination graph invariants on  Kneser graphs]{Several Roman domination graph invariants on  Kneser graphs}
\affiliation{
   Faculty of Natural Science and Mathematics,  University of Banja Luka, Bosnia and Herzegovina}
\keywords{ Kneser graphs, Roman domination,   total Roman domination,  signed Roman domination}
\begin{document}
\publicationdata
{vol. 25:1}{2023}{14}{10.46298/dmtcs.10506}{2022-12-19; 2022-12-19; 2023-03-16}{2023-04-25}
\maketitle
\begin{abstract}
	This  paper considers  the following three Roman domination graph invariants on Kneser graphs: 
	Roman domination, total  Roman domination, and signed Roman domination. 
	For Kneser graph $K_{n,k}$, we present  exact values for   Roman domination number $\gamma_{R}(K_{n,k})$ and total Roman domination number $\gamma_{tR}(K_{n,k})$ proving that for  $n\geqslant k(k+1)$, $\gamma_{R}(K_{n,k}) =\gamma_{tR}(K_{n,k}) = 2(k+1)$. For signed Roman domination number $\gamma_{sR}(K_{n,k})$,  the new lower and upper bounds for $K_{n,2}$ are provided: we prove that for $n\geqslant 12$, 
	the lower bound is equal to 2, while the upper bound depends on the parity of $n$ and  is equal to 3 if $n$ is odd, and equal to $5$ if $n$ is even. For graphs of smaller dimensions,  exact values are found by applying   exact methods from literature. 
\end{abstract}

\section{Introduction}

Let $G = (V,E)$ be a simple connected graph, with a set of vertices $V$, a set of edges $E$ and its order $|V|$. For an arbitrary vertex $v\in V$, open neighborhood $N(v)$ is defined as  set $\{u\in  V \mid uv\in  E\}$, while closed neighborhood  $N[v]$  is  set $N[v] = N(v)\cup \{v\}$.

\it The domination set $S$ \rm of  graph $G$ is defined as the subset of  set $V$ such that 
	\begin{equation*}
		(\forall u\in V\setminus S)(\exists v\in S)\ \  uv\in E.
\end{equation*} 
The minimum cardinality $\gamma(G)$ of a domination set is called \it the domination number  \rm of  graph $G$.   

\it Roman domination function \rm (RDF)  on  graph $G$, formally introduced by  \cite{coc04}, is defined as  function $f:V\to\{0,1,2\}$ which satisfies the condition
\begin{equation}
	\label{eq:c1rd}
	(\forall v\in V_0)(\exists u \in V_2) \  \ uv \in E,
\end{equation}
where $V_i=\{v\in V \mid f(v) = i\},\, i = 0,1,2$.
The  \it{weight} \rm of  function  $f$ is  value  $f(V) = \sum_{v\in V}f(v)$. The minimum value of the weights of all RDFs on graph $G$, denoted with $\gamma_R(G)$, is called the \it{Roman domination number }\rm  (RDN).

The basic relation between domination and Roman domination numbers is given in the following property.
\begin{prp} \cite{coc04} \label{rbou} For any graph $G$, it holds $\gamma(G) \le \gamma_R(G) \le 2 \cdot \gamma(G)$.
\end{prp} 

\it Total Roman domination function \rm (TRDF), introduced by \cite{liu2013roman}, is defined as  function $f:V\to\{0,1,2\}$, i.e., by the partition $(V_0,V_1,V_2)$ of  set $V$, which satisfies  conditions \eqref{eq:c1rd} and
\begin{equation}\label{eq:c2trd}
	(\forall v \in V)\     \sum_{u\in N(v)}f(u) \geqslant 1.
\end{equation}

In literature,  condition \eqref{eq:c2trd} is also introduced with an equivalent property that the subgraph of  graph $G$ induced with  vertices with a positive label has no isolated vertices.

\it{The total Roman domination number }\rm  (TRDN) of  graph $G$,  denoted with  $\gamma_{tR}(G)$, is   the minimum  weight  $f(V)= \sum_{v\in V}f(v)$ of all TRDFs $f$ on  $G$.
As each TRDF satisfies condition \eqref{eq:c1rd}, it is also an RDF. Therefore, the following observation is straightforward.
\begin{obs} \label{obs1} \cite{ahangar2016total} For each graph $G$ without isolated vertices, it holds $\gamma_R(G)\leqslant\gamma_{tR}(G)$. 
\end{obs}

\it Signed Roman domination function \rm (SRDF) is  a function $f:V\to\{-1,1,2\}$ for which it holds
\begin{equation}\label{eq:c1srd}
	(\forall v \in V_{-1})(\exists u \in V_2) \  \ uv \in E, 
\end{equation}
where
$V_i=\{v\in V \mid f(v) = i\},\, i \in\{ -1,1,2\}$ 
and
\begin{equation}\label{eq:c2srd}
	(\forall v \in V)\    \sum_{u\in N[v]}f(u) \geqslant 1.
\end{equation}

For proving new results, the following equivalent of the last condition is introduced. For $v\in V$, let  $\alpha_v,\beta_v$ and $\gamma_v$ represent cardinalities $|N(v)\cap V_2|,|N(v)\cap V_1|$ and $|N(v)\cap V_{-1}|$, respectively. Then  condition \eqref{eq:c2srd} is equivalent to  condition  \eqref{eq:v}
\begin{equation}\label{eq:v}
	(\forall v\in V) \ \  2\alpha_v+\beta_v-\gamma_v +f(v) \geqslant 1.
\end{equation}

\it{The signed Roman domination number }\rm  (SRDN)  $\gamma_{sR}(G)$  of  graph $G$ is the   minimum weight of all SRDFs  on graph $G$. 

The concept of \it the Kneser graph $K_{n,k},n,k\in \mathbb N$ \rm is introduced by \cite{kneser55}.
The set of vertices of  graph $K_{n,k}$ is  set of all $k$-element subsets of  set $\{1,2,\hdots,n\}$ and two vertices are adjacent if corresponding sets are disjoint.  Its order is $\binom n k$ and this is a type of regular graph with degree of each vertex equal to $\binom {n-k}k$. If $n<4$,  graph $K_{n,2}$ is edgeless. $K_{n,1}$ is the complete graph, while $K_{2k,k}$ for $k>1$ is not connected. Therefore,  in the rest of the paper we suppose that $n> 2k$ and $k> 1$. An example of the Kneser graph  for $n=5$ and $k=2$ ($K_{5,2}$) is given in Figure \ref{fig:kneser_example}. It can be noticed that this graph is isomorphic to  the Petersen graph.
\begin{figure}[h]
	\centering
	\includegraphics[height=180pt]{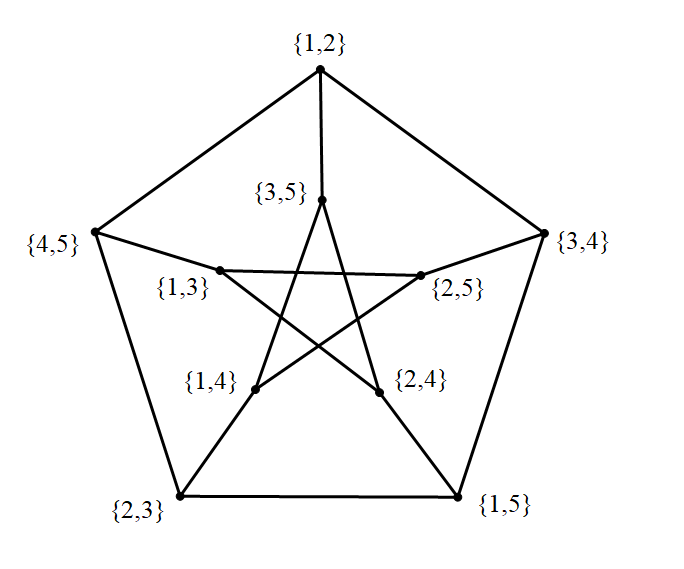}
	\label{fig:kneser_example}
	\caption{The Kneser graph $K_{5,2}$}
\end{figure}

\subsection{Previous work}

Graph $G$ is said to be a Roman graph if $\gamma_R(G) = 2\gamma(G)$. Several classes of Roman graphs were studied by \cite{coc04,henning2002characterization,yero2013roman,xueliang2009roman}. 
The exact result for the RDN of  generalized Petersen graphs was given by \cite{wang2011roman}.
Some more results regarding  RDN  can be found  in \cite{mobaraky2008bounds,liu2012upper,favaron2009roman,kartelj2021roman,li2021note}, for example.
A detailed review of  results on many variants of RDN is out of the scope of this paper and can be found in \cite{chellali2020varieties, chellali2021varieties}.

The relation between TRDN and (total) domination number as well as  with  RDN   was studied by \cite{martinez2020further,ahangar2016total}. Several bounds on SRDN in terms of  graph order, size, minimum and maximum vertex degree and (signed) domination number were explained by \cite{ahangar2014signed}. The authors also gave the exact value of a SRDN for some special graph classes: $\gamma_{sR}(K_3) = 2 \mathrm{\ and}\ \gamma_{sR}(K_n) = 1,n\neq 3,\ \gamma_{sR}(K_{1,n-1}) = 1,n=2l $ and $ \gamma_{sR}(K_{1,n-1}) = 2,n=2l+1,l\in\mathbb N,\ \gamma_{sR}(C_n) = \lceil\frac{2n}{3}\rceil,n\geqslant  3$ and $\gamma_{sR}(P_n) = \lfloor\frac{2n}{3}\rfloor,n\geqslant  1\ $.  The SRDN  was also considered for: digraphs  by \cite{sheikholeslami2015signed},  trees by \cite{henning2015signed}, the join of graphs by  \cite{behtoei2014signed} and planar graphs by \cite{zec2021signed}.

The value of the domination number for the Kneser graph is determined in Theorem \ref{dk}.  
\begin{thm}\cite{ostergaard2014bounds} \label{dk}  For $n \ge k \cdot (k+1)$, it holds $\gamma(K_{n,k}) = k+1$.
\end{thm}

Domination problems are quite an attractive research domain which 
has captivated researchers from various fields over the past few decades, including mathematicians and computer scientists. It is known that, for example, determining the Roman domination number in case of general graphs is NP--hard~\cite{coc04}. That implies that a successful application of provenly strong exact computational paradigms is not expected for arbitrary large graphs.  
Thus, general widely--applied exact methods, such as the branch-and-bound framework~\cite{lawler1966branch},  are usually restricted to a successful application onto small to middle-sized graphs. However, these techniques still serve here in several ways ($i$) to determine Roman domination-type numbers on small-sized graphs, and ($ii$) to get an insight into these numbers in case of some graph classes w.r.t.\  graph parameters. Please note that from the theoretical point of view, these techniques do not provide any proof on  established  Roman domination numbers. In this work, some exact methods based on Integer linear programming (ILP) techniques~\cite{graver1975foundations} are used for solving the corresponding problems  for some Kneser graphs of small dimensions.  More precisely, the model given by \cite{burger2013binary}  is used to obtain  the results presented in Remark  \ref{rkk2r} and Remark  \ref{rkk3r},   and the model exposed by \cite{filipovic2022solving} is used to obtain the  results presented in Remark  \ref{srkk2r}. The formulations of these ILP  models are given in Appendix~\ref{app:A}.

\section{New results for Kneser graphs}

\subsection{(Total) Roman domination for Kneser graphs}
In this section we present  exact values for (total) Roman domination numbers for Kneser graphs.

\begin{thm} \label{trkk} For $n \ge k \cdot (k+1), k>1$, it holds $\gamma_{tR}(K_{n,k}) = 2(k+1)$.
\end{thm}
\begin{proof}
	\textbf{\underline{Step 1}:\ { $\gamma_{tR} \geqslant 2(k+1)$.}}\\
	First, let us  show that for an arbitrary TRDF $\overline{f}$  it holds 
	\begin{equation} \label{eq:1}
		\overline{f}(V)\geqslant 2(k+1).
	\end{equation}
	Suppose that $\overline {f}$ is defined by  partition $(V_0,V_1,V_2)$.
	We consider all possible values for     $|V_2|$.
	
	{\bf Case 1:} $|V_2| = 0$.\\
	Since $\overline {f}$ is TRDF and $V_2$ is empty, then  $V_0$ is empty as well, so $\overline{f}(V)=|V_1|$. Therefore 
	\begin{align*}|V_1| & = |V| =  \binom{n}{k} = \frac{n(n-1)\hdots(n-k+1)}{k!}\\ 
		&\geqslant   \frac{(k^2+k)(k^2+k-1)\hdots(k^2+1)}{k!}>\frac{k^{2k}}{k!}.
	\end{align*}
	For $k = 2$ we have $\frac{k^{2k}}{k!} = 8>6=2(k+1)$. If $k\geqslant 3$, we get 
	$\frac{k^{2k}}{k!} =\frac{k^k}{k!}k^k\geqslant k^k\geqslant 2(k+1)$.
	So, in this case  $\overline{f}(V) \geqslant 2(k+1)$ holds.
	
	{\bf Case 2:}  $1\leqslant |V_2|\leqslant k-1.$ \\
	At most $|V_2|\cdot k$ different numbers (from  set $\{1,\ldots n\}$) are used to form vertices from  set $V_2$. 
	Then  at least $n-|V_2|\cdot k\geqslant k^2+k-(k-1)\cdot k=2k$ different numbers do not appear in any vertex from  set $V_2$. Let $X$ denote  set of these numbers.
	Let us identify $2k$ vertices not belonging to  set $V_2$ and  not adjacent to any vertex from set $V_2$.
	Let  $\{s_1,s_2,\ldots,s_{k-1}\}$ be  set of some   $k-1$ different numbers  which are chosen such that every vertex from $ V_2$   contains at least one of these numbers $s_i,\ i=1,\ldots,k-1$. An illustration of this procedure is shown in Example \ref{example1}.
	Now, let $Y= \{\{s_1,s_2,\ldots,s_{k-1},s\}: s\in X\}$. It holds that $Y\cap V_2 = \emptyset$. In addition, any vertex from $Y$  is not adjacent to any vertex from  set  $V_2$. Since  $\overline{f}$ is a TRDF, we conclude that  $Y\subseteq V_1$. 
	Therefore, $|V_1|\geqslant |Y|= 2k$, so using assumption $|V_2|\geqslant 1$, we get
	\begin{equation*}
		\overline{f}(V)=2|V_2|+|V_1|\geqslant 2|V_2|+2k\geqslant 2|V_2|+2k-2|V_2|+2=2(k+1).
	\end{equation*}
	
	{\bf Case 3:}   $ |V_2| = k.$\\
	Similarly, as in Case 2, at most  $|V_2|\cdot k = k^2$ different numbers are used to form vertices from  set $V_2$. Let us consider the numbers which do not appear in any vertex from set $V_2$ and denote set of such numbers with $X$. 
	Therefore, it holds  $|X|\geqslant n-k^2\geqslant k^2+k-k^2=k$. We here analyze two subcases.
	
	{\em Subcase 3.1:} 
	All vertices from set  $V_2$ are  adjacent to each other.\\
	Then,  by choosing one number per each vertex, we can identify total   $k^k$  vertices, such that neither of them is  adjacent to any vertex from $V_2$. Also notice that  none of these vertices belong to $V_2$.
	Therefore we conclude that all these vertices belong to  set $V_1$.
	Thus,  we get  $|V_1|\geqslant k^k\geqslant 2k$ and 
	\begin{equation*}
		\overline{f}(V)=2|V_2|+|V_1| \geqslant 2|V_2|+2k\geqslant 2|V_2|+2k-2|V_2|+2\geqslant 2k+2=2(k+1).
	\end{equation*}
	
	{\em Subcase 3.2:}  There exists at least one pair of non-adjacent vertices in  set $V_2$. \\
	Let $u,v\in V_2$ be vertices such that $u\cap v\neq \emptyset$ and $s_1\in u\cap v$. If we  choose numbers $s_2,s_3\hdots,s_{k-1}$ such that each of the remaining $(k-2)$ vertices from  set $V_2$ contains at least one of these numbers, then    set  $\{s_1,s_2,...,s_{k-1}\}$ has the same properties like the appropriate one from  Case 2.   Similarly as above, we conclude   that  each vertex of  form $\{s_1,s_2,\ldots,s_{k-1},s\}$, where $s\in X$ , belongs to  set $V_1$. Therefore, $|V_1|\geqslant k\geqslant 2$ and again 
	\begin{equation*}
		\overline{f}(V)=2|V_2|+|V_1|\geqslant 2k+2=2(k+1),
	\end{equation*} which concludes Case 3.
	
	{\bf Case 4:} $|V_2| \geqslant k+1$. \\
	In this case, it trivially holds that $\overline{f}(V)\geqslant 2(k+1)$.
	
	\textbf{\underline{Step 2}: { $\gamma_{tR} \leqslant2(k+1)$}.}\\
	Following the idea from \cite{ostergaard2014bounds}, we construct  function $f$ as follows. Let  $V_2$ be a collection of $k+1$ disjoint $k$-sets defined as
	\begin{equation*}
		V_2 = \{\{1,2,\ldots,k\},\{k+1,k+2,\ldots,2k\},\ldots,\{k^2+1,k^2+2,\ldots,k^2+k\}\}.
	\end{equation*}
	Let $V_1 = \emptyset$ and $V_0=V\backslash V_2$. The weight of the function $f$  is 
	$$f(V) = 2|V_2| = 2(k+1).$$
	Let us show that $f$ is a TRDF.
	Each vertex, i.e., $k$-element set from $V_0$ is non-disjoint with  at  most $k$ vertices from  set $V_2$, so it is disjoint with at least one vertex from $V_2$. Therefore,  condition  $\eqref{eq:c1rd}$ is satisfied.

	Given the previous consideration, for each vertex $v\in V_0$, $\sum_{u\in N(v)}f(u) \geqslant 1$ also holds. 
	From the construction of set $V_2$,  each vertex $v\in V_2$ is adjacent to   all other vertices from $V_2$. So, for all $v\in V_2$, it holds that $\sum_{u\in N(v)}f(u) = 2k\geqslant 1$. Condition \eqref{eq:c2trd} is thus satisfied, which concludes the proof.
	
\end{proof}

\begin{exm} \label{example1} By this example we illustrate the procedure shown in Case 2 of the previous theorem.
	
	Let $n=20,\ k=4$ and let $f$ be a TRDF for $K_{20,4}$, such that  
	\begin{equation*}
V_2=\{\{1,2,3,4\},\{5,6,7,8\},\{5,6,9,10\}\}.
	\end{equation*}
	
	Notice that $n = k^2+k$ and $|V_2| = k-1$.
	
	Let us take three different numbers: $s_1,s_2$, and $s_3$ such that  for each $v\in V_2$ it holds $v\cap\{s_1,s_2,s_3\}\neq\emptyset$. For example, let $s_1 = 1,s_2 = 5$, and $s_3 = 8$. 
	
	The vertices from  set $V_2$ contain $10$ different elements, which is less than $|V_2|\cdot k=12$. 
	
	Let   $X =\{13,14,\hdots,20\}$ and   $Y = \{\{s_1,s_2,s_3,s\}:s\in X\}\subset V$. It is obvious that for every two vertices $u\in Y$ and  $v\in V_2$, it holds $u\cap v\neq\emptyset$, i.e., in graph $K_{20,4}$ no vertex from set $Y$  has a neighbor from  set $V_2$.   Also, for each $u\in Y$, it holds that $u\notin V_2$. Therefore, since $f$ is TRDF, every such vertex must belong to  set $V_1$, so $|V_1|\geqslant |Y|= 8= 2k$.

\end{exm}

The result for the Roman domination number  follows straightforwardly. 

\begin{cor} \label{rkk} For $n \geqslant k \cdot (k+1), k>1$, it holds $\gamma_R(K_{n,k}) = 2(k+1)$.
\end{cor}
\begin{proof}

	\textbf{\underline{Step 1}:\ { $\gamma_{R}(K_{n,k}) \geqslant 2(k+1)$.}}

	One can notice that the complete proof for the lower bound of TRDN in Step 1 of  Theorem  \ref{trkk} is based only on using property \eqref{eq:c1rd}. 
	Since each RDF must also  satisfy that property, the same lower bound holds for RDN.

	\textbf{\underline{Step 2}: { $\gamma_{R}(K_{n,k}) \leqslant2(k+1)$}.}\\
	This inequality is a  straightforward consequence of  Theorem~\ref{trkk} and Observation~\ref{obs1}.

\end{proof}
\begin{obs}
	The inequality in Step 2 of Corollary~\ref{rkk}  also follows from Theorem \ref{rbou} and Property \ref{dk}.
\end{obs}

The following two remarks (Remark~\ref{rkk2r} and Remark~\ref{rkk3r}) contain results for  Kneser graphs $K_{n,2}$ and $K_{n,3}$, which are not  covered by Corollary \ref{rkk}.   As previously mentioned, we used the ILP model  from  \cite{burger2013binary} to find RDN  of these graphs.  The RDFs which correspond to these  solutions and  ILP model details are presented  in  Appendix \ref{app:A}.

\begin{rem} \label{rkk2r} It holds 
	$$\gamma_R(K_{5,2}) =  
	6.$$
\end{rem}

\begin{rem} \label{rkk3r} It holds
	$$\gamma_R(K_{n,3}) = \begin{cases} 
		14, & n=7,8,9, \\
		12, & n=10, \\
		10, & n=11.
	\end{cases}$$
	
\end{rem}

\subsection{Signed Roman domination for Kneser graphs}

In this section we present new lower and upper bounds for the signed Roman domination number for Kneser graphs $K_{n,2}$.

\begin{thm}\label{thm:signed} 
	For $n\geqslant 12 $ it holds:
	\begin{itemize}
		\item  $2\leqslant \gamma_{sR}(K_{n,2}) \leqslant 3, \  n \mathrm{\ is \ odd}$,
		\item $ 2\leqslant \gamma_{sR}(K_{n,2}) \leqslant 5, \ n \mathrm{\ is \ even}$.
	\end{itemize}
\end{thm}
\begin{proof}
	\textbf{\underline{Step 1}:}\ {$\gamma_{sR}(K_{n,2})\geqslant 2.$}\\
	Let $\overline{f} $ be an arbitrary SRDF, defined as $(V_{-1},V_1,V_2)$. Then for every vertex $v\in V$, inequality in condition  \eqref{eq:c2srd} holds. By summing up all the inequalities from condition \eqref{eq:c2srd}, we get
	
	\begin{equation} \label{eqn:sum}
		\sum_{v\in V} \sum_{u\in N[v]}\overline{f}(u) \geqslant \binom{n}{2}
	\end{equation}
	
	Since each vertex $v\in V$ has the degree  $\binom{n-2}{2}$, at the left hand side of inequality \eqref{eqn:sum},   value $\overline{f}(v)$ appears exactly  $\binom {n-2}{2}+1$ times. Thus,
	
	$$ \left(\binom {n-2}{2}+1\right)\sum_{v\in V}\overline{f}(v)\geqslant \binom n 2\ \ \Leftrightarrow 
	\left(\binom {n-2}{2}+1\right)\overline{f}(V)\geqslant \binom n 2\ \ \Leftrightarrow$$

	$$\overline{f}(V)\geqslant \frac{n(n-1)}{(n-2)(n-3)+2}.$$

	Expression $\frac{n(n-1)}{(n-2)(n-3)+2}$ is greater than $1$ and since the SRDN must be an integer, it holds that $\gamma_{sR}(K_{n,2})\geqslant 2$, which concludes the  proof of Step 1. \\
	\textbf{\underline{Step 2}}:\ 
	{ $ \gamma_{sR}(K_{n,2}) \leqslant
		3, \  n \mathrm{\ is \ odd}$ and $ \gamma_{sR}(K_{n,2}) \leqslant
		5,  n \mathrm{\ is \ even}$.}
	
	{\bf Case 1:} $n \mathrm{\ is \ odd} $.\\
	Let us partition the set $\{1,2,\hdots,n\}$ on sets $A_n$ and $B_n$ and the set $V$ on sets $A_{n,2}, B_{n,2}$ and $C_{n,2}$, as shown in Tab. \ref{tab14}.
	We introduce the function $f=(V_{-1},V_1,V_2)$, where sets $V_{-1},V_1$ and $V_2$ are given in the last three rows of Tab. \ref{tab14}. We show that  $f(V) = 3$ and $f$ is SRDF. 
	
	\begin{table}[h]
		\centering
		\begin{tabular}{|l |l|} 
			\hline 
			$A_n$ & $\{1,2,\hdots,\frac{n-3}{2}\}$ \\
			\hline
			$B_n$ &$ \{\frac{n-1}{2}, \frac{n+1}{2},\hdots,n\}$\\
			\hline
			$ A_{n,2}$ &  $\{\{a,b\}|a,b\in A_n\}$\\
			\hline
			$ B_{n,2}$ & $\{\{a,b\}|a,b\in B_n\}$\\
			\hline
			$ C_{n,2}$ & $ \{\{a,b\}|a\in A_n,b\in B_n\}$\\
			\hline
			$ V_2$ & $\{\{1,2\},\{2,3\},\{3,4\},\hdots,\{\frac{n-5}{2},\frac{n-3}{2}\},\{1,\frac{n-3}{2}\}\}$ \\
			\hline
			$ V_{-1}$ & $C_{n,2}$\\
			\hline
			$ V_1$ & $V\setminus (V_2\cup V_{-1})$\\
			\hline
			
		\end{tabular}
		\caption{\label{tab14} The construction of SDRF $f$ for which $f(V) = 3$ }
	\end{table}

	Notice that $V_2\subset A_{n,2}$ and  $V_1=(A_{n,2}\setminus V_2)\cup B_{n,2}$.

	It holds $|A_n| = \frac{n-3}{2},\ |B_n| = \frac{n+3}{2},\ |A_{n,2}| = \binom{(n-3)/2}{2}$, $|B_{n,2}| = \binom{(n+3)/2}2$ and $|C_{n,2}| = \frac{n-3}{2}\cdot\frac{n+3}{2} = \frac{n^2-9}{4}$. 
	
	We have $|V_2| =\frac{n-3}{2},\ |V_1| =\binom{(n-3)/2}{2}-\frac{n-3}{2}+\binom{(n+3)/2}{2}  = \frac{n^2-4n+15}{4} $ and $|V_{-1}| = \frac{n^2-9}{4} $, so $f(V) =2|V_2|+|V_1|-|V_{-1}| = 3. $ 
	
	Let us now prove that $f$ is an SRDF.

	Let $v = \{a,b\}\in V_{-1}$ be an arbitrary vertex. W.l.o.g. suppose that $a\in A_n$. From the definition of  sets $V_{-1}$ and $V_{2}$, it follows that $a$ occurs in exactly two vertices of  set $V_2$, so $v$ has  exactly
	$|V_2|-2=\frac{n-7}{2}>0$ neighbors labeled by $2$. The conclusion is that  condition \eqref{eq:c1srd} is satisfied.
	
	Let us now prove that  condition \eqref{eq:v} is satisfied.  
	
	\begin{enumerate}
		\item[(\it{i})] First let $v = \{a,b\}$ be an arbitrary vertex from  set $V_2$. 
		
		Notice that $a,b\in A_n$. From the definition of $V_2$, it follows that  $a$ and $b$ occur  in exactly $3$ vertices in  set $V_2 $, including vertex $v$. 
		
		Let $\{a,e\}$ and $\{b,f\}$ be the other two vertices from $V_2$ which contain $a$ and $b$, respectively. So
		$\alpha_v = |V_2|-|\{v,\{a,e\},\{b,f\}\}|=\frac{n-3}{2}-3 =\frac{n-9}{2}$. 
		
		To calculate $\beta_v$, we now observe those vertices from  set $V_1$ which are not adjacent to $v$, i.e., those which contain $a$ or $b$. These vertices are from  set $A_{n,2}\setminus V_2$ of the form $\{a,c\}$,  where $c\in A_n\setminus\{a,b,e\}$, or of the form $\{b,d\}$, where $d\in A_n\setminus\{a,b,f\}$. The total number of such vertices is $2\cdot\left(\frac{n-3}{2}-3\right)$.
		Therefore, $\beta_v = |V_1| - 2\cdot\left(\frac{n-3}{2}-3\right) =\frac{n^2-8n+51}{4} $.

		As  vertices from  set $V_{-1}$ which are not adjacent to $v$ are those of  form $\{a,c\}$ and $\{b,c\}$ for each $c\in B_n$,  vertex $v$ has $\gamma_v =|V_{-1}|-2\cdot\frac{n+3}{2} = \frac{n^2-4n-21}{4}$ neighbors in this set. 
		
		Finally, for  $v\in V_2$, it holds $2\alpha_v+\beta_v-\gamma _v +f(v)=  11$, the conclusion being that  condition \eqref{eq:v} is satisfied for vertices from  set $V_2$.
		
		\item[(\it{ii})] For $v = \{a,b\}\in V_1$, we have two possibilities: $v\in A_{n,2}\setminus V_2$ or $v\in B_{n,2}$.
		\begin{itemize}
			\item $v\in A_{n,2}\setminus V_2$.

			Here $a,b\in A_n$ and these elements are contained in exactly $4$ vertices, namely $\{a,e\},\{a,f\},\{b,g\}$ and $\{b,h\}$, which are all labeled with $2$.
			
			 This implies $\alpha_v = |V_2\setminus \{\{a,e\},\{a,f\},\{b,g\},\{b,h\}\}| = |V_2|- 4=\frac{n-11}{2}.$
			
			To calculate $\beta_v$ for this case, we again observe the vertices from  set $V_1$ which are not adjacent to $v$. Such vertices form  set 
			\begin{equation*}
               \{v\}\cup\{\{a,c\}|c\in  A_n\setminus\{a,b,e,f\}\}\cup\{\{b,c\}|c\in  A_n\setminus\{a,b,g,h\} \}. 
			\end{equation*}

			The cardinality of this set is equal to $1+2\cdot\left(\frac{n-3}{2}-4\right)=n-10$. Therefore, we get $\beta_v = |V_1|-(n-10)=\frac{n^2-8n+55}{4}$. The set of neighbors in  set $V_{-1}$ which are not adjacent to $v$ are of the form $\{a,c\}$ and $\{b,c\}$ for each $c\in B_n$. So, $\gamma_v = |V_{-1}|-2\cdot\frac{n+3}{2}=\frac{n^2-4n-21}{4}$.
			So, for this case, we conclude $2\alpha_v+\beta_v-\gamma _v +f(v)=9$.
			
			\item  $v\in B_{n,2}$.
			
			Here  $a,b\in B_n$, so neither $a$ or $b$ are  contained in any vertex from $V_2$, which gives $\alpha_v = |V_2|$. Neighbors of $v$ from $V_1$ form  set $(A_{n,2}\setminus V_2)\cup\{\{c,d\}|c,d\in B_n\setminus \{a,b\}\}$, with its cardinality equal to 
			\begin{equation*}
				\beta_v=\binom{(n-3)/2}{2}-|V_2|+\binom{(n+3)/2-2}{2}=\frac{n^2-8n+15}{4}.
			\end{equation*}
			
			Now it is left to calculate $\gamma_v$.
			All vertices labeled with $-1$, which are not adjacent to $v$  form  set $\{\{a,c\}|c\in A_n\}\cup \{\{b,c\}|c\in A_n\}$. This gives $\gamma_{v} = |V_{-1}|-2\cdot\frac{n-3}{2} = \frac{n^2-4n+3}{4}$. 
			
			Thus for this case we get  $2\alpha_v+\beta_v-\gamma _v +f(v)=1$. So, condition \eqref{eq:v} is satisfied for all vertices labeled by $1$.
		\end{itemize}
		
		\item[(\it{iii})] Let $v = \{a,b\}, a\in A_n, b\in B_n$ be an arbitrary vertex from  set $V_{-1}$. Let $\{a,e\}$ and $\{a,f\}$ be two vertices from $V_2$, which are not adjacent to $v$.  All other vertices from $V_2$ are  adjacent to $v$, so $\alpha_v = |V_2|-2 = \frac{n-7}{2}$. 
		Vertices from  $V_{1}$ which are not adjacent to $v$  form  set $\{\{a,c\}| c\in A_n\setminus \{a,e,f\}\}\cup \{\{b,c\}| c\in B_n\setminus\{b\}\}$. This gives $\beta_v =|V_1|-\left(\frac{n-3}{2}-3\right)-\left(\frac{n+3}{2}-1\right)=\frac{n^2-8n+31}{4}$. 
		To calculate $\gamma_v$, we consider the vertices from  set $V_{-1}$ which are not adjacent to $v$. These vertices form  set $\{v\}\cup\{\{a,c\}|c\in B_n\setminus\{b\}\}\cup\{\{b,c\}|c\in A_n\setminus\{a\}\}$. So, $\gamma_v =|V_{-1}|-\left(1+\left(\frac{n+3}{2}-1\right)+\left(\frac{n-3}{2}-1\right)\right)=\frac{n^2-4n-5}{4} $. Thus, if $v\in V_{-1}$, it holds 
		$2\alpha_v+\beta_v-\gamma _v +f(v)=1$, i.e.,  condition \eqref{eq:v} is satisfied if $v\in V_{-1}$.
	\end{enumerate}
	As we analyzed every possible case for $v\in V$, we conclude that $f$ satisfies condition \eqref{eq:v}.  Therefore, the proof of the theorem  for Case 1: $n \mathrm{\ is \ odd}$ is finished.
	
	{\bf Case 2:}  $n \mathrm{\ is \ even}$. \\
	We define SRDF $f$ for which $f(V)=5$.  We consider two subcases: $n \equiv 0\ (\mathrm{mod} \ 4) $ and $n \equiv 2\ (\mathrm{mod} \ 4) $.
	
	{\em Subcase 2.1: }\label{srd:subcase2.1} $n \equiv 0\ (\mathrm{mod} \ 4)$.
	
	\begin{table}[h]
		\centering
		\begin{tabular}{|l |l|} 
			\hline 
			$A_n$ & $\{1,2,\hdots,\frac{n-2}{2}\}$ \\
			\hline
			$B_n$ &$ \{\frac{n}{2}, \frac{n+2}{2},\hdots,n\}$\\
			\hline
			$ A_{n,2}$ &  $\{\{a,b\}|a,b\in A_n\}$\\
			\hline
			$ B_{n,2}$ & $\{\{a,b\}|a,b\in B_n\}$\\
			\hline
			$ C_{n,2}$ & $ \{\{a,b\}|a\in A_n,b\in B_n\}$\\
			\hline
			$ V_2$ & \makecell{$\{\{1,2\},\{3,4\},\hdots,\{\frac{n-6}{2},\frac{n-4}{2}\}\}\cup\{\{\frac{n}{2},\frac{n+2}{2}\},\{\frac{n+4}{2},\frac{n+6}{2}\},\hdots,$\\ $\{n-2,n-1\}\}\cup\{\{1,3\},\{2,4\}\},$}\\
			\hline
			$ V_{-1}$ & $ C_{n,2}\setminus\{\frac{n-2}{2},n\}$\\
			\hline
			$ V_1$ & $V\setminus (V_2\cup V_{-1})$\\
			\hline
			
		\end{tabular}
		\caption{\label{tab15} The construction of SDRF $f$ for which $f(V) = 5$ }
	\end{table}
	
	Similarly, as in Case 1 we introduce sets $A_n,B_n,A_{n,2},B_{n,2}$ and $C_{n,2}$, as well as function $f=(V_{-1},V_1,V_2)$, given in Tab. \ref{tab15}. 
	Notice that  
	$|A_n| = \frac{n-2}{2},\ |B_n| = \frac{n+2}{2},\ |A_{n,2}| = \binom{(n-2)/2}{2}$, $|B_{n,2}| = \binom{(n+2)/2}2,|C_{n,2}| = \frac{n-2}{2}\cdot\frac{n+2}{2} = \frac{n^2-4}{4}$.
	Also, $V_{1} =  (A_{n,2}\setminus V_2) \cup (B_{n,2}\setminus V_2)\cup \{\frac{n-2}{2},n\}$.
	
	We have $|V_2| =\frac{n-4}{4}+\frac{n}{4}+2 = \frac{n+2}{2},\ |V_{-1}|  = \frac{n^2-4}{4}-1 = \frac{n^2-8}{4} $ and  $|V_1|=\binom {n}{2} -\left(\frac{n+2}{2}+\frac{n^2-8}{4}\right) = \frac{n^2-4n+4}{4} $, so $f(V) =2|V_2|+|V_1|-|V_{-1}| = 5. $ 
	
	Let us now check whether the condition \eqref{eq:c1srd} is satisfied. 
	
	Let $v = \{a,b\}\in V_{-1}$ be an arbitrary vertex and w.l.o.g. suppose that $a\in A_n,\, b\in B_n$.
	Since $n\geqslant 12$,  $\{1,2\},\{3,4\}\in V_2$. 
	\begin{itemize}
		\item If $a\geqslant 3$, then  $\{1,2\}\cap v=\emptyset$, i.e.  $\{1,2\}$ and $v$ are adjacent. Therefore, $v$ has a neighbor in  set $V_2$, which implies that the  condition \eqref{eq:c1srd} is satisfied. 
		
		\item If $a\leqslant 2$, then  $\{3,4\}\cap v=\emptyset$. Similarly, we conclude the condition \eqref{eq:c1srd} is satisfied.
	\end{itemize}
	Let us prove that condition \eqref{eq:v} is satisfied.
	
	For $n = 12$,   the values  $2\alpha_v+\beta_v-\gamma_v+f(v)$  are calculated for all vertices and the results are shown in Tab. \ref{tabb1}. From the last column of Tab. \ref{tabb1} one can see that  condition   \eqref{eq:v} holds.

	\begin{table}[h]
		\centering
		\small
		\begin{tabular}{|l |c|c|c|c|c|c|} 
			\hline
			$v$ & $\alpha_v$& $\beta_v$ & $\gamma_v$ & $f(v)$ & $2\alpha_v+\beta_v-\gamma_v+f(v)$  \\ 
			\hline
			$v\in V_2\cap A_{12,2}$&4&21&20&2&11\\
			\hline
			$v\in V_2\cap B_{12,2}$&6&15&24&2&5\\
			\hline
			$\{1,4\},\{2,3\}$&3&22&20&1&9\\
			\hline
			$\{1,5\},\{2,5\},\{3,5\},\{4,5\}$&5&19&21&1&9\\
			\hline
			$\{5,12\}$&7&14&24&1&5\\
			\hline
			$\{a,b\}\in V_1\cap B_{12,2}, b\neq 12$&5&19&21&1&9\\
			\hline
			$\{a,12\}\in V_1\cap B_{12,2}, a\neq 5$&6&14&25&1&2\\
			\hline
			$\{a,b\}\in V_{-1}\cap C_{12,2},a\neq 5, b\neq 12$&4&18&23&-1&2\\
			\hline
			$\{a,12\}\in V_{-1}\cap C_{12,2}$&5&16&24&-1&1\\
			\hline
			$\{5,b\}\in V_{-1}\cap C_{12,2}$&6&15&24&-1&2\\
			\hline
		\end{tabular}
		\caption{\label{tabb1} The values  $2\alpha_v+\beta_v-\gamma_v+f(v)$   for all vertices of graph $K_{12,2}$}
	\end{table}

	Let now $n\geqslant 16$. 
	
	The proof that condition \eqref{eq:v} is satisfied is similar to the corresponding proof in Case 1. It should be noted that in this case there are more subcases depending on the definitions of sets $V_2,V_1$ and $V_{-1}$. For that reason we shortened the proof, still covering all possible cases.
	
	\begin{enumerate}
		\item[(\it{i})] Let $v = \{a,b\}\in V_2$. We consider two possibilities.
		\begin{itemize}
			\item $v\in A_{n,2}$.  
			
			The lowest value for  $\alpha_v$ is obtained  for 
			$v\in \{\{1,2\},\{3,4\},\{1,3\},\{2,4\}\}$ and it is equal to $|V_2|-3 = \frac{n-4}{2}$. This also shows that  the minimum  value of $\beta_v$ is obtained for $v\notin \{\{1,2\},\{3,4\},\{1,3\},\{2,4\}\}$ and it is equal to $|V_1|- (n-6) = \frac{n^2-8n+28}4$.  For each $v\in A_{n,2}$, we have $\gamma_v = \frac{n^2-4n-16}{4}$. 
			
			Now $2\alpha_v + \beta_v - \gamma_v + f(v) \geqslant 2\cdot \frac{n-4}{2} + \frac{n^2-8n+28}4 - \frac{n^2-4n-16}{4} +2 = 9.$

			\item $v\in B_{n,2}$.
			
			For each $v\in B_{n,2}$, it holds that $\alpha_v = \frac{n}{2}$,  $\beta_v = \frac{n^2-8n+12}{4}$ and $\gamma_v = \frac{n^2-4n}{4}$. 
			
			Therefore, $2\alpha_v + \beta_v - \gamma_v + f(v) = 5.$
		\end{itemize}

		We hereby showed that for each $v\in V_2$, the inequality from condition \eqref{eq:v} is satisfied.
		
		\item[(\it{ii})]   For $v = \{a,b\}\in V_1$ we observe three possibilities.

		\begin{itemize}
			\item $v\in  A_{n,2}\setminus V_2$.
			
			The lowest value for  $\alpha_v$ is obtained  for  $v\in\{\{1,4\},\{2,3\}\}$ and it is equal to $|V_2|-4 = \frac{n-6}{2}$.
			
			Further,  the lowest value for $\beta_v$ is obtained when either  $a$ or $b$ belong to  set $A_n\setminus\{1,2,3,4,\frac{n-2}2\}$ and the other one is equal to $\frac{n-2}2$. 
			
			Here we get $\beta_v=|V_1|- \left(\frac{n-2}2-2+\frac{n-2}2-2\right)-1 = \frac{n^2-8n+24}4$.
			
			The greatest value for $\gamma_v$ is obtained when one of the numbers $a$ or $b$ is equal to $\frac{n-2}2$ and   $\gamma_v = |V_{-1}|-\left(2\cdot\frac{n+2}2-1\right)= \frac{n^2-4n-12}{4}$. 
			
			Thus, $2\alpha_v + \beta_v - \gamma_v + f(v) \geqslant 2\cdot \frac{n-6}{2} + \frac{n^2-8n+24}4 - \frac{n^2-4n-12}{4} +1 = 4.$

			\item $v\in  B_{n,2}\setminus V_2$. 
			
			For $a, b\in B_n\setminus\{n\}$, we get: $\alpha_v=|V_2|-2 = \frac{n-2}{2}$, $\beta_v = |V_1|- \left(\frac{n+2}2-2+\frac{n+2}{2}-3\right) = \frac{n^2-8n+16}4$ and
			$\gamma_v =|V_{-1}|-2\cdot\frac{n-2}2= \frac{n^2-4n}{4}$, so $2\alpha_v + \beta_v - \gamma_v + f(v)  = 3.$
			
			If one of $a$ or $b$ equals $n$ then: $\alpha_v=|V_2|-1 = \frac{n}{2}$, $\beta_v = |V_1|- \left(\frac{n+2}2-2+\frac{n+2}{2}-2\right) -1= \frac{n^2-8n+8}4$ and
			$\gamma_v =|V_{-1}|-\left(2\cdot\frac{n-2}2-1\right)= \frac{n^2-4n+4}{4}$, so $2\alpha_v + \beta_v - \gamma_v + f(v)  = 2.$

			\item $v =\{\frac{n-2}2,n\}$.  
			
			For this vertex we get: $\alpha_v= |V_2| = \frac{n+2}{2}$,  $\beta_v=|V_1|-\left(\frac{n-2}{2}-1+\frac{n+2}{2}-1\right) -1 = \frac{n^2-8n+8}{4}$,
			$\gamma_v = |V_{-1}|- \left(\frac{n+2}{2}-1+\frac{n-2}{2}-1\right) = \frac{n^2-4n}{4}$, which gives
			$2\alpha_v + \beta_v - \gamma_v + f(v)= 5.$
		\end{itemize}
		
		We hereby proved  that the inequality from condition \eqref{eq:v} is satisfied for $v\in V_1$.

		\item[(\it{iii})] For  $v = \{a,b\}\in V_{-1}$, we  also consider three cases.
		\begin{itemize}
			\item $a\in A_n\setminus\{\frac{n-2}2\} $ and $b\in B_n\setminus \{n\}$.  
			
			The smallest value for $\alpha_v$ is obtained for $a\in\{1,2,3,4\}$,  where $a$ and $b$ occur in exactly three vertices in  set $V_2$, so   $\alpha_v= |V_2|-3 = \frac{n-4}{2}$. 
			
			The smallest value of $\beta_v$ is achieved for $a\notin\{1,2,3,4\}$ and it is equal to $\beta_v=|V_1|-\left(\frac{n-2}{2}-2+\frac{n+2}{2}-2\right)  = \frac{n^2-8n+20}4$. 
			
			For each vertex $v$, in this case we get  $\gamma_v = |V_{-1}|- \left(\frac{n+2}{2}+\frac{n-2}{2}-1\right) = \frac{n^2-4n-4}{4}$. 
			
			Therefore, $2\alpha_v + \beta_v - \gamma_v + f(v) \geqslant 2\cdot \frac{n-4}{2} + \frac{n^2-8n+20}4 - \frac{n^2-4n-4}{4} -1 = 1.$

			\item $a = \frac{n-2}2$ and $b\in B_n\setminus\{n\}$.   
			
			In this case $v$ is not adjacent to only one vertex from $V_2$ which contains $b$, so   $\alpha_v= |V_2|-1 = \frac{n}{2}$. 
			
			Further, we get $\beta_v=|V_1|-\left(\frac{n-2}{2}-1+\frac{n+2}{2}-2\right)-1  = \frac{n^2-8n+12}4$ 
			
			and  $\gamma_v = |V_{-1}|- \left(\frac{n+2}{2}-1-\frac{n-2}{2}-1\right) = \frac{n^2-4n}{4}$.
			
			Therefore,	$2\alpha_v + \beta_v - \gamma_v + f(v) = 2.$
			
			\item $a\in A_n\setminus\{\frac{n-2}2\}$ and $b = n$.

			If  $a\in \{1,2,3,4\}$, we get:  $\alpha_v=\frac{n-2}{2}$, $\beta_v = \frac{n^2-8n+16}{4} $, $\gamma_v = \frac{n^2-4n}{4}$ and $2\alpha_v + \beta_v - \gamma_v + f(v)=1$.
			
			If  $a\notin \{1,2,3,4\}$, then holds:  $\alpha_v=\frac{n}{2}$, $\beta_v = \frac{n^2-8n+12}{4} $, $\gamma_v = \frac{n^2-4n}{4}$ and $2\alpha_v + \beta_v - \gamma_v + f(v)=2$.
		\end{itemize}
		
		This proves that the inequality from condition  \eqref{eq:v} is satisfied for $v\in V_{-1}$. 
	\end{enumerate}
	Since we covered all possible cases for  $n \equiv 0\ (\mathrm{mod} \ 4)$, the constructed function $f$ is an SDRF and this part of the theorem is proved.
	
	{\em Subcase 2.2:} $n \equiv 2\ (\mathrm{mod} \ 4)$.\\ 
	Let  sets $A_n,B_n,A_{n,2},B_{n,2}$ and $C_{n,2}$ be constructed as  in Subcase 2.1.
	The definition of function $f=(V_{-1},V_1,V_{2})$  such that $f(V)=5$ is given in Tab \ref{tab16}.
	
	\begin{table}[h]
		\centering
		\begin{tabular}{|l |l|} 
			\hline 
			$ V_2$ & 	\makecell[c]{$ \{\{1,2\},\{3,4\},\hdots,\{\frac{n-4}{2},\frac{n-2}{2}\}\}\cup\{\{\frac{n}{2},\frac{n+2}{2}\},\{\frac{n+4}{2},\frac{n+6}{2}\},\hdots,$\\ $\{n-1,n\}\}\cup\{\{1,3\},\{2,4\},\{\frac n 2,\frac{n+4}2\}\},$ }\\
			\hline
			$ V_{-1}$ & $C_{n,2}$\\
			\hline
			$ V_1$ & $V\setminus (V_2\cup V_{-1})$\\
			\hline
			
		\end{tabular}
		\caption{\label{tab16} The construction of SDRF $f$ for which $f(V) = 5$ }
	\end{table}

	Notice that $V_{1}=  (A_{n,2}\setminus V_2) \cup (B_{n,2}\setminus  V_2) $.
	
	The cardinalities of these sets are equal to:  $|V_2| =\frac{n-2}{4}+\frac{n+2}{4}+3 = \frac{n+6}{2},\ |V_{-1}|  = \frac{n-2}{2}\cdot\frac{n+2}{2} = \frac{n^2-4}{4} $ and  $|V_1|=\binom {n}{2} -\left(\frac{n+6}{2}+\frac{n^2-4}{4}\right) = \frac{n^2-4n-8}{4} $. Thus, $f(V) =2|V_2|+|V_1|-|V_{-1}| = 5. $

	Let us prove that  condition  \eqref{eq:c1srd} is satisfied.
	
	Since  $n\geqslant 14$, we have that $\{1,2\},\{3,4\}\in V_2$. Therefore,  similarly as in Subcase 2.1. it can be shown that    condition  \eqref{eq:c1srd} holds.

	Let us now prove that  condition  \eqref{eq:v} is satisfied. 
	
	For $n=14$ the  values $2\alpha_v + \beta_v - \gamma_v + f(v), v\in V$ are given in  Tab. \ref{tabb2}.  One can see that condition \eqref{eq:v} holds in this case.
	\begin{table}[h]
		\centering
		\begin{tabular}{|l |c|c|c|c|c|} 
			\hline
			$v$ & $\alpha_v$& $\beta_v$ & $\gamma_v$ & $f(v)$ & $2\alpha_v+\beta_v-\gamma_v+f(v)$  \\ 
			\hline
			$\{1,2\},\{1,3\},\{2,4\},\{3,4\}$&7&27&32&2&11\\
			\hline
			$\{5,6\}$&7&27&32&2&11\\
			\hline
			$\{7,8\},\{9,10\}$&8&22&36&2&4\\
			\hline
			$\{11,12\},\{13,14\}$&9&21&36&2&5\\ 
			\hline
			$\{7,9\}$&7&23&36&2&3\\ 
			\hline
			$\{1,4\},\{2,3\}$&6&28&32&1&9\\  
			\hline
			$\{a,b\}, a\in\{1,2,3,4\},b\in\{5,6\}$&7&27&32&1&10\\ 
			\hline
			$\{a,b\},a,b\in B_{14}\setminus \{7,9\}$&8&22&36&1&3\\ 
			\hline
			$\{a,b\}\in V_1\cap B_{14,2}, a\in\{7,9\}$&7&23&36&1&2\\ 
			\hline
			$\{a,b\}, a\in\{1,2,3,4\},b\in\{7,9\},$ &6&25&35&-1&1\\   
			\hline
			$\{a,b\}, a\in\{1,2,3,4\},b\in B_{14}\setminus\{7,9\},$   &7&24&35&-1&2\\  
			\hline
			$\{a,b\},a\in\{5,6\},b\in B_{14}\setminus\{7,9\},$ &8&23&35&-1&3\\   
			\hline
			$\{5,7\},\{5,9\},\{6,7\},\{6,9\}$  &7&24&35&-1&2\\  
			\hline

		\end{tabular}
		\caption{\label{tabb2} The values  $2\alpha_v+\beta_v-\gamma_v+f(v)$   for all vertices of graph $K_{14,2}$}
	\end{table}
	
	Let now $n\geqslant 18$.
	\begin{enumerate}
		\item[(\it{i})] Let $v = \{a,b\}\in V_2$. Similar to the previous subcase, we differ two cases.
		\begin{itemize}
			\item $v\in A_{n,2}$.  
			
			The minimum  value of $\alpha_v$ is obtained  for 
			
			$v\in \{\{1,2\},\{3,4\},\{1,3\},\{2,4\}\}$, where $\alpha_v = |V_2|-3 = \frac{n}{2}$. 
			
			This implies that  the minimum  value of $\beta_v$ is obtained for
			\\ $v\notin \{\{1,2\},\{3,4\},\{1,3\},\{2,4\}\}$ and  equals $\beta_v=|V_1|- 2\cdot\left(\frac{n-2}2-2\right) = \frac{n^2-8n+16}4$. 
			
			Further, we get $\gamma_v = |V_{-1}|-2\cdot\frac{n+2}2=\frac{n^2-4n-12}{4}$. 
			
			This gives $2\alpha_v + \beta_v - \gamma_v + f(v) \geqslant 2\cdot \frac{n}{2} + \frac{n^2-8n+16}4 - \frac{n^2-4n-12}{4} +2 = 9.$
			
			\item $v\in B_{n,2}$.
			
			In this case $\alpha_v$ is minimal   for vertex  
			$v\in \{\{\frac n 2,\frac{n+4}2\}\}$ for which  $\alpha_v = |V_2|-3 = \frac{n}{2}$. 
			
			The value $\beta_v$ is minimal for $a,b\notin \{\frac n 2,\frac{n+4}2\}$ for which $\beta_v=|V_1|- 2\cdot\left(\frac{n+2}2-2\right) = \frac{n^2-8n}4$. 
			
			Further, we get $\gamma_v = |V_{-1}|-2\cdot\frac{n-2}2=\frac{n^2-4n+4}{4}$. 
			
			This gives $2\alpha_v + \beta_v - \gamma_v + f(v) \geqslant 2\cdot \frac{n}{2} + \frac{n^2-8n}4 - \frac{n^2-4n+4}{4} +2 = 1.$
			
		\end{itemize}

		Therefore the inequality from condition \eqref{eq:v} is fulfilled for every $v\in V_2$.

		\item[(\it{ii})]  For $v\in V_1$ we consider two cases.

		\begin{itemize}
			\item $v\in  A_{n,2}\setminus V_2$. 
			
			Value $\alpha_v$ is the smallest for  $v\in\{\{1,4\},\{2,3\}\}$ and equals  $|V_2|-4 = \frac{n-2}{2}$.
			
			The minimum value of $\beta_v$ is obtained when  $a,b \in A_n\setminus\{1,2,3,4\}$ and it equals $\beta_v=|V_1|- \left(\frac{n-2}2-2+\frac{n-2}2-3\right) = \frac{n^2-8n+20}4$.
			
			We also get that $\gamma_v = |V_{-1}|-2\cdot\frac{n+2}2= \frac{n^2-4n-12}{4}$. 
			
			So in this case $2\alpha_v + \beta_v - \gamma_v + f(v) \geqslant 2\cdot \frac{n-2}{2} + \frac{n^2-8n+20}4 - \frac{n^2-4n-12}{4} +1 = 7.$

			\item $v\in  B_{n,2}\setminus V_2$. 
			
			Value  $\alpha_v$ is minimal when $a$ or $b$ belong to  set $\{\frac n 2,\frac{n+4}2\}$, where  $\alpha_v=|V_2|-3 = \frac{n}{2}$. 
			
			The lowest value of $\beta_v$ is obtained for $a,b\notin \{\frac n 2,\frac{n+4}2\}$, when  $\beta_v = |V_1|- \left(\frac{n+2}2-2+\frac{n+2}2-3\right) = \frac{n^2-8n+4}4$.
			
			Here it holds that $\gamma_v = |V_{-1}|-2\cdot\frac{n-2}2= \frac{n^2-4n+4}{4}$. 
			
			Therefore, $2\alpha_v + \beta_v - \gamma_v + f(v) \geqslant 2\cdot \frac{n}{2} + \frac{n^2-8n+4}4 - \frac{n^2-4n+4}{4} +1 = 1.$
			
		\end{itemize}
		The conclusion is  the inequality from condition \eqref{eq:v} holds for each $v\in V_1$.
		
		\item[(\it{iii})] For an arbitrary vertex $v = \{a,b\}\in V_{-1}$, we get the following results:
		
		\begin{itemize}
			
			\item If $a\in \{1,2,3,4\}$ and $b\in\{\frac n 2,\frac{n+4}2\}$ :  $\alpha_v=\frac{n-2}{2}$,\\ $\beta_v=|V_1|-\left(\frac{n-2}{2}-3+\frac{n+2}{2}-3\right)  = \frac{n^2-8n+16}4$, 
			$\gamma_v = \frac{n^2-4n}{4}$, 
			$2\alpha_v + \beta_v - \gamma_v + f(v)=1$.
			\item If $a\in \{1,2,3,4\}$ and $b\notin\{\frac n 2,\frac{n+4}2\}$ :  $\alpha_v=\frac{n}{2}$, \\ $\beta_v=|V_1|-\left(\frac{n-2}{2}-3+\frac{n+2}{2}-2\right)  = \frac{n^2-8n+12}4$,
			$\gamma_v = \frac{n^2-4n}{4}$, $2\alpha_v + \beta_v - \gamma_v + f(v)=2$.
			\item If $a\notin \{1,2,3,4\}$ and $b\in\{\frac n 2,\frac{n+4}2\}$ :   $\alpha_v=\frac{n}{2}$,  \\ $\beta_v=|V_1|-\left(\frac{n-2}{2}-2+\frac{n+2}{2}-3\right)  = \frac{n^2-8n+12}4$,
			$\gamma_v = \frac{n^2-4n}{4}$, $2\alpha_v + \beta_v - \gamma_v + f(v)=2$.
			\item If $a\notin \{1,2,3,4\}$ and $b\notin\{\frac n 2,\frac{n+4}2\}$ :  $\alpha_v=\frac{n+2}{2}$, \\ $\beta_v=|V_1|-\left(\frac{n-2}{2}-2+\frac{n+2}{2}-2\right)  = \frac{n^2-8n+8}4$,
			$\gamma_v = \frac{n^2-4n}{4}$, $2\alpha_v + \beta_v - \gamma_v + f(v)=3$.
			
		\end{itemize}
		It  follows  that the inequality from condition \eqref{eq:v} holds for each $v\in V_{-1}$. Therefore, the function $f$ introduced in this subcase is also an SRDF, which finally proves the theorem.
		
	\end{enumerate}
\end{proof}

We used the ILP model  from \cite{filipovic2022solving}  to find SRDN for some special cases of Kneser graphs which are provided in Remark \ref{srkk2r}.  The SRDFs which correspond to these solutions and ILP model details are presented  in Appendix \ref{app:A}.

\begin{rem} \label{srkk2r} It holds 
	$$\gamma_{sR}(K_{n,2}) = \begin{cases}
		5, & n=5,6,7,8, \\
		4, & n=10, \\
		3, & n=9,11.
	\end{cases}$$
\end{rem}
It can be observed that 	$\gamma_{sR}(K_{9,2}) =\gamma_{sR}(K_{11,2}) = 3$, which is in line with the proposed  upper bound proposed in Theorem~\ref{thm:signed} for odd $n$. Also, $\gamma_{sR}(K_{8,2})=5$, which is equal to the upper bound for graphs with greater even dimensions, considered in Theorem~\ref{thm:signed}.

\section{Conclusions}

This article  considered  the (total) Roman domination problem  for Kneser graphs $K_{n,k}$, $n\geqslant k(k+1)$ and the signed Roman domination problem for $K_{n,2}$. We proved that  $\gamma_{tR}(K_{n,k}) = \gamma_{R}(K_{n,k})=2(k+1)$, if $n\geqslant k(k+1)$. For all $n\geqslant 12$   the lower and upper bounds for SRDN were given for even $n$,   $ 2\leqslant \gamma_{sR}(K_{n,2}) \leqslant 5$, while for odd $n$,   $2\leqslant \gamma_{sR}(K_{n,2}) \leqslant 3$.

Finding a more tighter bounds for SDRNs in cases $k=2,3$, or even the exact values could be a promising direction for future work.
Also, finding the bounds for (T)RDN, when $2k< n<k(k+1)$, as well as the bounds of SRDN for $k\geqslant 3$ remains open.
Investigating the  other graph invariants on Kneser graphs, such as Roman $k-$domination \cite{kammerling2009roman}, double Roman domination  \cite{beeler2016double}, signed double Roman domination \cite{ahangar2019signed}, strong Roman domination \cite{alvarez2017strong}, etc. could be a challenge for further work. 

\section*{Acknowledgments}
This research is  supported by a bilateral project between Austria  and Bosnia and Herzegovina funded by the Ministry of Civil Affairs of Bosnia and Herzegovina  under no. 1259074.

\appendix
\section{Results on small Kneser graphs}\label{app:A}
\subsection{Results on small Kneser graphs for RDP}

The  ILP model from \cite{burger2013binary}  was  implemented in Cplex solver ~\cite{lima2010ibm}   to obtain the RDN of some Kneser graphs of small sizes. It is stated as follows.
The set of variables is defined by:

\begin{equation*}
	x_v=
	\begin{cases}
		1, &f(v) = 1,\\
		0, &otherwise.
	\end{cases}
\end{equation*}

\begin{equation*}
	y_v=
	\begin{cases}
		1, &f(v) = 2,\\
		0, &otherwise.
	\end{cases}
\end{equation*}

The ILP model for RDP  is formulated as:
\begin{align*}
	&\min \sum_{v\in V}(x_v+2y_v)\\
	&\mbox{s.t.}\\
	&x_v+y_v + \sum_{u\in N(v)}y_u\geqslant 1, \forall \ v\in V,\\
	&x_v+y_v\leqslant 1, \ \forall v\in V,\\
	&x_v, y_v \in \{0,1\}, \ \forall v \in V.
\end{align*}

In Tab. \ref{tab1} we  present the obtained ILP solutions for RDFs with minimum weight. The first two columns contain basic parameters for  graph $K(n,k)$. The third column contains value of  RDN obtained by solving the corresponding ILP model. The last three columns contain detailed information about   sets $(V_2,V_0,V_1)$, respectively, which corresponds to the exact solution obtained by the ILP model.

\begin{table}[h]
	\centering
	\small
	\begin{tabular}{|l |l|l|p{7cm}|c|c|} 
		\hline
		$n$ & $k$& $f(V)$ & $V_2$ & $V_0$ & $V_1$  \\ 
		\hline
		\hline
		$4,5$ & $2$ &$6$ & $\{\{1,2\},\{1,3\},\{2,3\}\}$ & $ V\setminus V_2$ & $\emptyset$\\
		\hline
		$6$ & $3$ &$20$ &  $\{\{1,2,3\},\{1,2,4\},\{1,2,5\},\{1,2,6\},\{1,3,4\},$ $\{1,3,5\},\{1,3,6\},\{2,3,4\},\{2,3,5\},\{2,3,6\}\}$ & $ V\setminus V_2$ & $\emptyset$\\
		\hline
		$7,8,9$ & $3$ &$14$ &  $\{\{1,2,5\},\{1,3,6\},\{1,4,7\},\{2,3,4\},\{2,6,7\},$ $\{3,5,7\},\{4,5,6\}\}$ & $ V\setminus V_2$ & $\emptyset$\\ 
		\hline 
		$10$ & $3$ &$12$ & $\{\{1,2,8\},\{1,4,8\},\{2,4,10\},\{3,5,9\},\{3,6,7\}$ $\{5,6,9\}\}$ & $ V\setminus V_2$ & $\emptyset$\\ 
		\hline
		$11$ & $3$ &$10$ & \shortstack {$\{\{1,5,9\},\{1,7,9\},\{2,3,8\},\{4,5,7\},\{6,10,11\}\}$} & $ V\setminus V_2$ & $\emptyset$\\ 
		\hline
	\end{tabular}
	\caption{\label{tab1} The solutions obtained by solving the ILP  on small Kneser graphs}
\end{table}

\subsection{Results on small Kneser graphs for SRDP}

The  ILP model from \cite{filipovic2022solving}  was  implemented in Cplex solver  ~\cite{lima2010ibm}  to obtain relation between TRDN and domination number as well as with RDN  values of SRDN  for small Kneser graphs. It is stated as follows. 

The set of variables is given by: 
\begin{equation*}
	x_v=
	\begin{cases}
		1, &f(v) = 1\\
		0, &otherwise.
	\end{cases}
\end{equation*}
\begin{equation*}
	y_v=
	\begin{cases}
		1, &f(v) = 2\\
		0, &otherwise.
	\end{cases}
\end{equation*}
The ILP model for SRDP is formulated as:

\begin{align*}
	&\min  \sum_{v \in V} {\left(2 x_v + 3 y_v -1\right)}\\
	&\mbox{s.t.}\\
	&x_v  + y_v\leq 1, \ \forall u \in V,\\
	&x_v  + y_v +\sum_{u \in N(v)}{y_u} \geq 1,\ \forall v \in V,\\
	&\sum_{u \in N[v]}{\left(2 x_u + 3 y_u -1\right)} \geq 1,\ \forall v \in V,\\
	&x_v, y_v \in \{0,1\},\  \forall v \in V.
\end{align*}

Tab. \ref{tab2} contains the SRDFs of the minimum weight which are  obtained by solving the aforementioned ILP model for SRDP on small Kneser graphs $K(n,2)$. The table is organized similarly as Tab.~\ref{tab1}, with the exception that column $k$ is omitted since  $k=2$ in all cases. The last three columns carry the information about  sets $V_2,V_1$, and $V_{-1}$, respectively, in the corresponding partition.
\begin{table}[h]
	\centering
	\small
	\begin{tabular}{|l|l|p{3.5cm}|p{4cm}|l|} 
		\hline
		$n$ & $f(V)$ & $V_2$ & $V_1$ & $V_{-1}$  \\ 
		\hline
		\hline
		$4$  & $3$ & $\{\{1,2\},\{1,3\},\{2,3\}\}$ & $ \emptyset$ & $V\setminus V_2$\\
		\hline
		$5$  & $5$ & $\{\{1,3\},\{1,4\},\{3,4\}\}$ & $ \{\{2,4\},\{2,5\},\{4,5\}\}$ & $V\setminus(V_2\cup V_1)$\\
		\hline
		$6$  & $5$ & $\{\{1,2\},  \{1,4\},  \{1,5\},$  $\{2,5\},  \{3,6\},  \{4,5\} \}$ & $ \{\{2,4\}\}$ & $V\setminus(V_2\cup V_1)$\\
		\hline
		$7$  & $5$ & $\{\{2,5\},  \{2,6\},  \{3,4\} $, $\{5,6\} \}$ & $\{\{1,2\},  \{1,5\}, \{1,6\},  \{1,7\},$   $\{2,7\},  \{5,7\},  \{6,7\}  \}$ & $V\setminus(V_2\cup V_1)$\\
		\hline
		$8$  & $5$ & $\{\{1,7\},  \{2,5\},  \{2,8\},$  $\{3,4\}, \{3,6\},  \{4,6\},$ $  \{5,8\}\}$ & $\{1,2\},  \{1,5\}, \{1,8\}, $  $\{2,7\},  \{5,7\},  \{7,8\}  \}$ & $V\setminus(V_2\cup V_1)$\\
		\hline
		$9$  & $3$ & \shortstack {$\{ \{3,4\}, \{3,8\},  \{4,8\}   \}$ }& $ \{\{1,2\},  \{1,5\}, \{1,6\},  \{1,7\},$  $\{1,9\},  \{2,5\},  \{2,6\}, \{2,7\},$ $ \{2,9\},  \{5,6\},  \{5,7\},  \{5,9\},$ $   \{6,7\},  \{6,9\},  \{7,9\}\} $& $V\setminus(V_2\cup V_1)$\\
		\hline
		$10$  & $4$ & $\{\{1,2\},  \{3,5\},  \{4,8\},$ $ \{6,7\},  \{6,10\},  \{7,9\},$ $ \{9,10\}  \}$ & $ \{\{1,3\},  \{1,4\},  \{1,5\},  \{1,8\},$ $ \{2,3\},  \{2,4\},  \{2,5\},  \{2,8\},$ $ \{3,4\},  \{3,8\},  \{4,5\},  \{5,8\},$ $\{6,9\},  \{7,10\}\}$ & $V\setminus(V_2\cup V_1)$\\
		\hline
		
		$11$ & $3$ & $\{ \{3,6\}, \{3,11\},  \{4,6\},$  $\{4,11\}    \}$ & $ \{\{1,2\},  \{1,5\}, \{1,7\},  \{1,8\},$ $ \{1,9\}, \{1,10\},  \{2,5\},  \{2,7\},$ $\{2,8\},  \{2,9\},\{2,10\},\{3,4\},$ $\{5,7\},  \{5,8\},  \{5,9\}, \{5,10\},$ $ \{6,11\},  \{7,8\},  \{7,9\},  \{7,10\}, $ $\{8,9\},  \{8,10\},\{9,10\}\}$ & $V\setminus(V_2\cup V_1)$\\
		\hline
	\end{tabular}
	\caption{\label{tab2} The solutions obtained by solving the ILP on small Kneser graphs}
\end{table}

\pagebreak

\nocite{*}
\bibliographystyle{abbrvnat}
\bibliography{paperKfinal}
\label{sec:biblio}

\end{document}